\documentclass{aims}

\usepackage{paralist}
\usepackage{amssymb}
\usepackage{amsthm}

\usepackage{amsmath}
  \usepackage{paralist}
  \usepackage{graphics} 
  \usepackage{epsfig} 
\usepackage{graphicx}  \usepackage{epstopdf}
 \usepackage[colorlinks=true]{hyperref}
\hypersetup{urlcolor=blue, citecolor=red}

\def\currentvolume{X}
 \def\currentissue{X}
  \def\currentyear{200X}
   \def\currentmonth{XX}
    \def\ppages{X--XX}
     \def\DOI{10.3934/xx.xx.xx.xx}

\newtheorem{theorem}{Theorem}[section]
\newtheorem{corollary}{Corollary}[section]

\newtheorem{lemma}{Lemma}[section]

\theoremstyle{definition}
\newtheorem{definition}[theorem]{Definition}
\newtheorem{remark}{Remark}[section]

\newtheorem{exmp}{Example}[section]

\let \al=\alpha
\let \be=\beta
\let \var=\varphi
\let \vare=\varepsilon

\let \de=\delta
\let \th=\theta
\let \la=\lambda

\let \ga=\gamma

\let \q=\quad

\let \med=\medskip
\let \smal=\smallskip
\let \dps =\displaystyle

\newcommand{\R}{\mathbb{R}}
\newcommand{\N}{\mathbb{N}}

\makeatletter\def\theequation{\arabic{section}.\arabic{equation}}

\title[On Stability for  Impulsive  Delay Differential Equations] 
      {On Stability for  Impulsive  Delay Differential Equations and Application to a Periodic  Lasota-Wazewska Model}

\author[Teresa Faria and Jos\'{e} J. Oliveira]{}


 \email{teresa.faria@fc.ul.pt}
 \email{jjoliveira@math.uminho.pt}

\thanks{$^*$Corresponding author: Tel: +351253604084, e-mail:jjoliveira@math.uminho.pt}

\begin{document}
\maketitle

\centerline{\scshape Teresa Faria}
\medskip
{\footnotesize
 \centerline{Departamento de Matem\'atica and CMAF-CIO,}
   \centerline{ Faculdade de Ci\^encias, Universidade de Lisboa}
   \centerline{Campo Grande, 1749-016 Lisboa, Portugal}
} 

\medskip

\centerline{\scshape Jos\'{e} J. Oliveira$^*$}
\medskip
{\footnotesize
 \centerline{CMAT and Departamento de Matem\'atica e Aplica\c c\~oes,}
   \centerline{Escola de Ci\^encias, Universidade do Minho,}
   \centerline{Campus de Gualtar, 4710-057 Braga,
   	Portugal}
}

\bigskip

 \centerline{(Communicated by the associate editor name)}

\begin{abstract}
 We consider a class of scalar delay differential equations  with impulses and satisfying an Yorke-type condition, for which some 
 criteria for the global stability of the zero solution are  established. Here, the usual requirements about the impulses are relaxed.
 The  results can  be applied to  study  the stability of other solutions, such as periodic solutions. As an illustration, a very general periodic Lasota-Wazewska model with impulses and multiple time-dependent delays is addressed, and the global attractivity of its positive periodic solution analysed. 
 Our results are discussed within the context of recent literature.
\end{abstract}

\section{Introduction}\label{seccao-intruducao}\setcounter{equation}{0}

The   high impact of differential equations with
impulses  in terms of their application in population dynamics, disease modelling and other fields has lead to an increasing interest in the theory of impulsive systems. Recently, the theoretical analysis of existence and regularity of solutions to  impulsive systems with delays, as well as  the study of  concrete impulsive models used in mathematical biology, have become
an important area of research.

In this paper, we study  a class of scalar impulsive  differential equations  with an instantaneous negative feedback term and (possibly unbounded) delays. 

Let $\tau:[0,\infty)\to[0,\infty)$ be
a continuous function     such that    $\dps\lim_{t\to\infty}(t-\tau(t))=\infty$.  Without loss of generality, we 
suppose that $t\mapsto t-\tau(t)$ is non-decreasing; otherwise, whenever necessary, we can replace $t\mapsto t-\tau(t)$ by $d(t):=\inf_{s\ge t} (s-\tau(s))$, which is non-decreasing and satisfies the above conditions for $t-\tau(t)$.
For $t\geq0$, denote by  $PC(t)=PC([-\tau(t), 0];\R)$  the space of real functions that are piecewise continuous functions on $[-\tau(t), 0]$  and left continuous  on $(-\tau(t), 0]$,
with the norm
$\|\phi\|_t:=\dps\sup_{\theta\in[-\tau(t),0]}|\phi(\theta)|$ for $\phi\in PC(t)$. 

We consider  a scalar impulsive delay differential equation (DDE) of the form
\begin{equation}\label{1.1)}
\begin{split}
&x'(t)+a(t)x(t)=f(t, x_t),\q  0\le t\ne t_k, \\
&\Delta(x(t_k)):=x(t_k^+)-x(t_k) =I_k(x(t_k)),\q \, k=1,2,\dots,
\end{split}
\end{equation}
where: $x'(t)$ is the left-hand derivative of $x(t)$, 
$(t_k)_k$ is an increasing  sequence of positive numbers with $t_k\to \infty$, 
$a:[0,\infty)\to[0,\infty)$ and  $I_k:\R\to\R$ are  continuous functions; 
$x_t$ denotes the restriction of $x(t)$ to the interval $[t-\tau(t),t]$, with
$x_t\in PC(t)$  given by
$$x_t(\th)=x(t+\th)\q {\rm for}\q -\tau(t)\le \th\le 0;$$
$f(t,\var)$ is a functional defined for $t\ge 0$ and $\var\in PC(t)$ with some regularity discussed below.
We  also assume that $f(t,0)= 0$ for $t\ge 0$ and $I_k(0)=0$ for $k\in\N$, thus $x\equiv 0$ is a solution of \eqref{1.1)}.

The particular case of \eqref{1.1)} with $a(t)\equiv 0$ reads as
\begin{equation}\label{1.2)}
\begin{split}
&x'(t)=f(t, x_t), \q  0\le t\ne t_k, \\
&\Delta(x(t_k))=I_k(x(t_k)),\q \, k=1,2,\dots,
\end{split}
\end{equation}
and  has  been studied by many authors (see e.g. \cite{3.,14.,16.,17.,18.,19.,ZhaoYan}). In the present study, the aim however is to take full advantage of the negative instantaneous feedback given by the term $a(t)x(t)$ on the left-hand side of \eqref{1.1)}.

\med

A rigorous abstract formulation of the existence of solutions problem for \eqref{1.1)},  or  for more general impulsive DDEs, has been well established in the  literature, and will be omitted here (see e.g.,   \cite{1.,LiuB,LiuB2,8.,14., 15.} and references therein for more details).  We need however to set some  properties for $f$, as well as  clarify the space of initial conditions.

For a compact interval $[\al,\be]\subset \R$, denote by $B([\al,\be];\R)$ the space of bounded functions from $[\al,\be]$ to $\R$  equipped with  the supremum norm, and  $PC([\al,\be];\R)$ the subspace of $B([\al,\be];\R)$ of functions that are piecewise continuous on $[\al,\be]$  and left continuous  on $(\al,\be]$.   
Now, define the space $PC=PC((-\infty,0];\R)$ as the space of functions from $(-\infty,0]$ to $\R$ for which the restriction to each compact interval $[\alpha,\beta]\subset(-\infty,0]$  is in the closure of   $PC([\al,\be] ;\R)$ in $B([\al,\be];\R)$. 
A function $\varphi\in PC$ is continuous everywhere except at most for an enumerable number of points $s$ for which $\varphi(s^-)$, $\varphi(s^+)$ exist and $\varphi(s^-)=\varphi(s)$. Denote by $BPC$ the subspace of all bounded functions in $PC$, $BPC=\{\varphi\in PC:\var \ {\rm is\ bounded\ on\ }(-\infty,0]\}$, with the supremum norm $\|\varphi\|=\sup_{s\leq0}|\varphi(s)|$.  Clearly, each space $PC(t)$ can be taken as a subspace of $BPC$. For equations without impulses, the subspaces of continuous functions $C, BC$ and $C(t)$ of $PC, BPC$ and $PC(t)$, respectively, will be considered.

For $t\ge 0, \var\in BPC$, we 
write $f(t,\var_0)=f(t, L(t,\var))=:F(t,\var),$ where $L(t,\var)=\var_0$, i.e., $L(t,\var)=\var_{|_{[-\tau(t),0]}}$. For the impulsive DDE \eqref{1.1)}, we consider initial conditions of the form
\begin{equation}\label{1.3)}
x(t_0+s)=\varphi(s),\q s\le 0,
\end{equation}
with $t_0\geq0,\ \varphi\in BPC.$
In view of our purposes, we may suppose that the extension  $F:[0,\infty)\times BPC\to\R$ of $f$ is continuous or piecewise continuous (for simplicity,  we abuse the language and refer to $f$ as being  continuous or piecewise continuous as well), but more general frameworks are allowed. For $f$ piecewise continuous, under  very general conditions (which are satisfied with the hypotheses we shall  imposed in Section \ref{seccao-estabilidade}),
from \cite{1.,8.,13.,15.} it follows that  the initial value problem \eqref{1.1)}-\eqref{1.3)} has a unique  solution $x(t)$ defined on $[t_0,\infty)$. This solution will be denoted by $x(t,t_0,\varphi)$. 

\med

As an important example of a one-dimensional DDE appearing in mathematical biology, we refer to the well-known Lasota-Wazewska equation
$$
N'(t)=-a N(t)+be^{-\be N(t-\tau)},\q t\ge 0
$$
($a,b,\be,\tau>0$) introduced in \cite{11.} to model the survival of read blood cells.
Generalisations of this equation with periodic coefficients and multiple delays have received  much attention from mathematicians and other researchers (see  \cite{4.,7.} and references therein).
More recently  (\cite{5., 6., 12.}),    impulses have been added  to such models,  as in
\begin{equation}\label{1.4)}
\begin{split}
&N'(t)+a(t)N(t)=\sum_{i=1}^n b_i(t)e^{-\be_i (t)N(t-\tau_i(t))},\q  0\le t\ne t_k, \\
&\Delta N(t_k):=N(t_k^+)-N(t_k) =I_k(N(t_k)),\q   k=1,2,\dots,
\end{split}
\end{equation}
where all the coefficients and delays are periodic functions with a common period $\omega>0$ and $0<t_1<t_2<\dots$ with $t_{k+p}=t_k+\omega,\, k=1,2,\dots $ for some positive integer $p$.

The investigation in this paper was partially inspired by the works of  Tang  \cite{9.},  Yan  \cite{13.} and Zhang  \cite{19.}. Another strong motivation to study the stability of impulsive models of the form \eqref{1.1)} was to apply criteria of stability for such models to the  Lasota-Wazewska impulsive system \eqref{1.4)}, and obtain generalisations of Liu and Takeuchi's result in  \cite{6.}.
\med

Often the entire space BPC is not a suitable set of initial conditions, and more restrictive sets should be considered.
A set $S\subseteq BPC$ is called an {\it admissible set of initial conditions} if
$$
\varphi\in S\Rightarrow x_t(\cdot,t_0,\varphi)\in S,\q t\geq t_0.
$$
For models from mathematical biology as \eqref{1.4)}, clearly only positive solutions are meaningful, and therefore admissible.
In this paper we establish sufficient conditions for the stability of the zero solution of the impulsive DDE \eqref{1.1)}.
These results can  be applied to  study  the stability of other solutions, such as periodic solutions. We recall here some stability definitions for an admissible set of initial conditions $S\subseteq BPC$.

\begin{definition} Let $f(t,0)= 0$ for $t\ge 0$ and $I_k(0)=0,\, k\in\N$. 
	We say that the solution $x\equiv0$ of \eqref{1.1)} is {\bf stable} in $S$ if for any $\varepsilon>0$ and $t_0\geq0$, there exists $\delta=\delta(t_0,\varepsilon)>0$ such that
	$$
	||\varphi||<\delta\Rightarrow|x(t,t_0,\varphi)|<\varepsilon,\q{\rm for }\q t\geq t_0,\, \varphi\in S.
	$$
	The solution $x\equiv0$ of \eqref{1.1)} is said to be {\bf asymptotically stable} in $S$ if it is stable and for any $t_0\geq0$, there exists $\delta=\delta(t_0)>0$ such that
	$$\varphi\in S,\, 
	||\varphi||<\delta\Rightarrow|x(t,t_0,\varphi)|\to 0\q {\rm as \q }t\to\infty.
	$$
	The solution $x\equiv0$ of \eqref{1.1)} is said to be {\bf globally attractive} in $S$ if all solutions of \eqref{1.1)} with initial conditions in $S$ tend to zero as $t\to\infty$.
	Finally,   the solution $x\equiv0$ of \eqref{1.1)} is  {\bf globally asymptotically stable} if it is stable and global attractive.   If it is well understood which set $S$ we are dealing with, we omit  the reference to it.
\end{definition}

The remainder of this paper is organized in two sections. Section \ref{seccao-estabilidade} deals with the stability and global asymptotic stability of the  zero solution of the impulsive DDE \eqref{1.1)}. First, a main set of  assumptions for \eqref{1.1)} is introduced, and a brief comparison with other hypotheses considered in the literature is presented. Sufficient conditions  for  the global attractivity of zero are established by treating  separately non-oscillatory  and   oscillatory solutions of \eqref{1.1)}.
In Section \ref{seccao-modelo-LasotaWazewska}, the global asymptotic stability of  a positive  $\omega$-periodic solution to \eqref{1.4)} is studied by using the results in Section \ref{seccao-estabilidade}. The particular case of constant delays $\tau_i(t)=m_i\omega$ for $m_i$ positive integers ($1\le i\le n$), is further analysed; in this situation, better results are obtained when the impulses  in \eqref{1.4)} are given by linear  functions  $I_k(u)=b_ku$. A comparison of our criteria with recent results in the literature is also included.

\section{Stability}\label{seccao-estabilidade}
\setcounter{equation}{0}

In this section, we address the stability and global attractivity of the trivial solution of \eqref{1.1)} in $BPC$, but another admissible set  of initial conditions $S\subseteq BPC$ can be chosen.

The main assumptions that will be  imposed are taken from the ones described below. Hypothesis (H3)  will be chosen in alternative to  (H2).
Occasionally,  weaker  versions of these assumptions  will  be considered.

\begin{itemize} 
	
	\item[(H1)] 
	there exist positive sequences $(a_k)$ and $(b_k)$ such that
	\begin{equation}\label{2.1} b_kx^2\leq x[x+I_k(x)]\leq a_kx^2,\q x\in\R,\ k\in\N;
	\end{equation}
	
	\item[(H2)]  (i) the sequence $P_n=\dps \prod_{k=1}^{n}a_k$ is bounded;
	(ii) $\dps \int_0^{\infty} a(u)\, du=\infty$;

	\item[(H3)] (i)  the sequence $P_n=\dps \prod_{k=1}^{n}a_k$ is convergent; 
	\item[] (ii)  if $w:[0,\infty)\to\R$ is a bounded, non-oscillatory and piecewise differentiable function with $w'(t)w(t)\le 0$  on $(t_k,t_k+1)$, $k\in\N$, and
	$\dps\lim_{t\to\infty}w(t)=c\neq0$, then
	$$ \int_0^\infty f(s,w_s)\, ds=-\rm{sgn} (c)\infty;$$
	
	\item[(H4)] there exist piecewise continuous functions
	$\lambda_1,\lambda_2:[0,\infty)\to[0,\infty)$
	such that
	\begin{equation}\label{2.2} 
	-\lambda_1(t){\mathcal M}_t(\varphi)\leq f(t,\varphi)\leq\lambda_2(t){\mathcal M}_t(-\varphi),\q t\geq0,\, \varphi\in PC(t),
	\end{equation}
	where ${\mathcal M}_t(\varphi):=
	\dps\max\left\{0,\sup_{\theta\in[-\tau(t),0]}\varphi(\theta)\right\}$
	is the Yorke's functional on $PC(t)$;

	\item[(H5)]  there exists $T>0$ with $T-\tau(T)\ge 0$ such that 
	$$\alpha_1\alpha_2<1,
	$$    
	where the coefficients
	$\al_i:=\al_i(T)$ are
	given by
	\begin{equation}\label{2.3} 
	\alpha_i=\sup_{t\geq T}
	\int_{t-\tau(t)}^t\lambda_i(s)e^{-\int_s^ta(u)\, du}B(s)\, ds,\q i=1,2,
	\end{equation}
	with    $B(t):=\dps\max_{\theta\in[-\tau(t),0]}\bigg(\prod_{k:t+\theta\leq t_k<t}b_k^{-1}\bigg)$.
	
\end{itemize} 

We observe that  the hypotheses (H1) and (H4) imply $I_k(0)=0$  and $f(t,0)=0$ for $k\in\N,\, t\ge 0$, thus $x=0$ is an equilibrium point of \eqref{1.1)}. In (H5) above, we make use of  the standard convention that the product $B(t)$ is equal to one if the number of factors is zero.

In the remainder of this paper, we shall use the notation
\begin{equation}\label{2.4} 
A(t)= \int_0^{t} a(u)\, du,\q t\ge 0. \end{equation}
We shall consider either the set of  conditions (H2), or 
alternatively  the requirements in  (H3). Condition (H2)(ii) translates as $\lim\limits_{t\to\infty} A(t)=\infty$ and is fulfilled in many interesting models in the literature; in this case, instead of (H3)(i) it will be sufficient to impose (H2)(i).
However, it is useful to consider the alternative hypothesis (H3)(ii), which in particular allows  dealing with \eqref{1.2)}.  
The constraint on the impulses given by  (H1)  implies in particular that $x(t_k^-)x(t_k^+)>0$ if $x(t_k^-)\ne 0$, with lower and upper bounds $b_k\le x(t_k^+)/x(t_k^-)\le a_k$ for $k=1,2,\dots$.

We now compare our hypotheses with the ones in the literature, in particular  in  references  \cite{13.,19.}, two 
major sources of inspiration for the analysis in this section.
As far as we know, the hypotheses on the impulses, (H1) and either (H2)(i) or (H3)(i), are novel and strongly relax the usual requirements in the  literature on stability for impulsive DDEs under Yorke-type conditions. In fact, typically either $I_k$ are supposed to be linear functions, or condition \eqref{2.1} is assumed with $a_k=1$ for all $k\in\N$, i.e.,
\begin{equation}\label{2.5} 
b_kx^2\leq x[x+I_k(x)]\leq x^2,\q x\in\R,\, k\in\N,
\end{equation} 
which implies $|x(t_k^+)|\le |x(t_k^-)|$, hence forces the solutions to approach zero at each instant of impulse  \cite{3.,HuoLiLui,13.,16.,17.,19.}.

Zhang  \cite{19.} treated only the case of system \eqref{1.2)}, and proved the global attractivity of its zero solution  provided that the impulsive functions $I_k$ and $f$ satisfy \eqref{2.5},  (H3)(ii), (H4),  and  the additional
generalised ``${3\over 2}$-type condition":
\begin{equation}\label{2.6} 
\al_1 \al_2<(3/2)^2
\end{equation} 
where $\al_1,\al_2$ are as in \eqref{2.3} with $a(t)\equiv 0$. 
In  \cite{13.}, Yan considered \eqref{1.1)} with a set of more restrictive assumptions: again, the impulsive functions $I_k$ were subject to condition \eqref{2.5},  the Yorke condition (H4) was imposed with $\la_1(t)=\la_2(t)=:\la(t)$ and an extra condition  to deal with non-oscillatory solutions was added; moreover,
instead of hypothesis (H5) or \eqref{2.6}, for the global attractivity of the trivial solution of \eqref{1.1)} Yan imposed the restriction
\begin{equation}\label{2.7} 
\sigma:=\sup_{t\geq 0}
\int_{t-\tau(t)}^t\lambda(s)e^{\int_{t-\tau(t)}^sa(u)du}B(s)\, ds < {3\over 2}\, ,
\end{equation} 
where    $B(t)$ is defined as in (H5). 
In the case  $\la_1(t)=\la_2(t)=\la(t)$, it is clear that $\al_1=\al_2\le \sigma$ for $\al_1,\al_2$ given by \eqref{2.3}, with $\al_1=\al_2< \sigma$ if $a(t)\not\equiv 0$; there are however positive functions $a(t)$ for which the condition $\sigma< {3\over 2}$ is less restrictive than $\al_1\al_2<1$. In this situation, it would be convenient to achieve stability results  under hypotheses similar to or less restrictive than \eqref{2.7}. This will be the subject of a forthcoming paper. Nonetheless we should emphasise that  the main idea  here was to take full advantage of the negative feedback term $a(t)x(t)$, rather than working with a ${3\over 2}$-type condition. This kind of approach has also been taken for non-impulsive DDEs: we shall refer later in this section to the work of  Tang  \cite{9.} (see also  \cite{7., 10.} for some alternative criteria), where  the DDE $x'(t)+ca(t)x(t)=f(t, x_t)$, with the functions $a(t), \tau(t), f(t,x_t)$ as in \eqref{1.1)} and $c$ a positive constant, was studied assuming (H2)(ii) and the following Yorke condition:
$$-a(t){\mathcal M}_t(\varphi)\leq f(t,\varphi)\leq a(t){\mathcal M}_t(-\varphi),\q t\geq0,\, \varphi\in C(t).
$$

We start our analysis with an auxiliary result from  \cite{13.}.
Let  $x(t)$ be a solution of \eqref{1.1)} on $[0,\infty)$ and define $y(t)$  by
\begin{equation}\label{2.8} 
y(t)=\prod_{k:0\leq t_k<t}J_k(x(t_k))x(t),
\end{equation} 
where, for each $k\in\N$,
$$
J_k(u):={u\over {u+I_k(u)}},\q u\in\R\setminus\{0\}.
$$

\begin{lemma}\label{lem2.1} 
	\cite{13.}  If $x(t)$ is a solution of \eqref{1.1)} on $[0,\infty)$, then the function $y(t)$ defined by \eqref{2.8} is a continuous function satisfying
	\begin{equation}\label{2.9} 
	y'(t)+a(t)y(t)=\prod_{k:0\leq t_k<t}J_k(x(t_k))f(t,x_t),\q t\geq0,\ t\neq t_k.
	\end{equation} 
\end{lemma}

To prove the global asymptotic stability of the trivial solution, we  consider separately oscillatory and non-oscillatory solutions to \eqref{1.1)}. We recall that  a solution $x(t)$  is oscillatory if it is not eventually zero and 
has arbitrarily large zeros;  otherwise, $x(t)$ is non-oscillatory.
First, we establish  criteria about the asymptotic behaviour of all non-oscillatory solutions.

\begin{lemma}\label{lem2.2} 
	Assume (H1),  (H2)(i) and \vskip 0.1in
	(H4*) for $t\geq0$ and $\varphi\in PC(t)$,
	$ f(t,\varphi)\leq0$ if $\varphi \geq0$ and $ f(t,\varphi)\geq0$ if $\varphi \leq 0.$
	\vskip 0.1in
	\noindent  Then, all non-oscillatory solutions of \eqref{1.1)} are bounded. If in addition (H2)(ii) holds,
	then    all non-oscillatory solutions of \eqref{1.1)}  converge to zero as $t\to\infty$.
\end{lemma}

\begin{proof} 
	From  (H1), we have
	\begin{equation}\label{2.10} 
	a_k^{-1}\leq J_k(u)\leq b_k^{-1}\q{\rm  for }\q u\neq0,\,\,k\in\N.
	\end{equation} 
	Take  a solution $x(t)$ of \eqref{1.1)} and let $y(t)$ be defined by \eqref{2.8}.
	From \eqref{2.10}, we have
	\begin{equation}\label{2.11} 
	|x(t)|=|y(t)|\prod_{k:0\leq t_k<t}J_k^{-1}(x(t_k))\leq|y(t)|\prod_{k:0\leq t_k<t}a_k,
	\end{equation} 
	For  a non-oscillatory solution $x(t)$, assume that $x(t)>0$ for $t\gg 0$ (the situation is analogous if $x(t)<0$ for $t\gg 0$). 
	Then $y(t)>0$ for large $t$ and, from \eqref{2.9} and (H4*), $y'(t)\le y'(t)+a(t)y(t)\leq0$ for  $t\gg0, t\neq t_k$. In particular, there are $c_0,w\ge 0$ such that $y(t)\searrow c_0$ and  $e^{A(t)}y(t)\searrow w$ as $t\to\infty$, where $A(t)$ is as in \eqref{2.4}. On the other hand, from (H2)(i) and  \eqref{2.11} we have
	$0<x(t)\le My(t)$ for $ t\gg 0$, where  $M>0$ is such that $P_k=\prod_{i=1}^ka_i\le M$ for $k\in\N$. Consequently $x(t)$ is bounded.
	Moreover, if  (H2)(ii) holds we deduce that $\lim_{t\to \infty}x(t)= \lim_{t\to \infty}y(t)=0$.
\end{proof}

When (H2)(ii) is not satisfied, the convergence to zero of non-oscillatory solutions is obtained if (H3)(ii) is imposed and, rather than having $P_k$ simply bounded,  $P_k$ is required to be convergent.

\begin{lemma}\label{lem2.3} 
	Assume (H1), (H4*), and either (H2) or (H3). 
	Then, all non-oscillatory solutions of \eqref{1.1)} converge to zero as $t\to\infty$.
\end{lemma}

\begin{proof} As above and without loss of generality, suppose that  $x(t)$ is an eventually positive solution of \eqref{1.1)}, and let $y(t)$ be defined by \eqref{2.8}. From the proof of Lemma \ref{lem2.2},
	$y(t)\searrow c_0,\ e^{A(t)}y(t)\searrow w$ as $t\to\infty$, with $c_0\ge 0$. If $\lim_{t\to\infty} A(t)=\infty$, by Lemma \ref{lem2.2} we have $\lim_{t\to\infty} x(t)=0$, hence we may suppose that 
	$\lim_{t\to\infty} A(t)=A_0\in [0,\infty)$, and that (H3) is satisfied.  Let $\lim P_k=\ga$. 
	
	For $t\gg 0$,
	$0<x(t)\leq\Big(\prod_{k:0\leq t_k<t}a_k\Big)y(t),
	$ 
	and by (H3)(i)
	$$
	c_+:=\limsup_{t\to\infty}x(t)\leq\gamma c_0,\q c_-:=\liminf_{t\to\infty}x(t)\geq0.
	$$
	We now show that $c_+=c_-=0$.

	If $\ga=0$, clearly $c_+=c_-=0$. Now, suppose $\ga >0$. Since $x(t)>0$ for large $t$, by \eqref{1.1)} and (H4*) we have $x'(t)\leq0$ on $(t_k,t_{k+1})$ for large $k$, which implies that there exist subsequences $(t_{n_1(k)})_k$ and $(t_{n_2(k)})_k$ of $(t_k)_k$, with $n_2(k)\leq n_1(k)$, such that $x\left(t^+_{n_1(k)}\right)\to c_+$ and $x\left(t_{n_2(k)}\right)\to c_-$. From \eqref{1.1)} and (H1), we have
	\begin{equation*}
	\begin{split}
	x\left(t_{n_1(k)}^+\right)& = x\left(t_{n_1(k)}\right)+I_{n_1(k)}\Big(x(t_{n_1(k)})\Big)  \\
	&\leq a_{n_1(k)}x\left(t_{n_1(k)}\right)\leq a_{n_1(k)}x\big(t^+_{n_1(k)-1}\big),
	\end{split}
	\end{equation*}      
	and, by iteration,
	\begin{equation}\label{2.12}
	x\left(t_{n_1(k)}^+\right) \leq \left(\prod_{i=n_2(k)}^{n_1(k)}a_i
	\right)x\left(t_{n_2(k)}\right).
	\end{equation}      
	Since
	$$
	\prod_{i=n_2(k)}^{n_1(k)}a_i=P_{n_1(k)}P_{n_2(k)-1}^{-1}\to\gamma\gamma^{-1}=1\q{\rm as}\q n\to\infty,
	$$
	from \eqref{2.12} it follows that $c_+\leq c_-$, and consequently $c_+=c_-:=c$. Hence $x(t)\to c\geq0$ as $t\to\infty$. If $c>0$,  define $\varepsilon:=\dps\inf_{n\in\N}\left(\prod_{i=1}^na_i^{-1}\right)>0$. From \eqref{2.9} and \eqref{2.10}, we obtain
	\begin{equation}\label{2.13}
	\begin{split}
	y(t)e^{A(t)}&=y(0)+\int_0^te^{A(s)}\Big (\prod_{k:0\leq t_k<s}J_k(x(t_k))\Big )f(s,x_s)\, ds\\
	&\leq y(0)+\vare \int_{0}^te^{A(s)}f(s,x_s)\, ds.
	\end{split}
	\end{equation}      
	But by (H3)(ii), $\int_0^\infty f(s,x_s)\, ds=-\infty$, hence  $\int_0^\infty e^{A(s)}f(s,x_s)\, ds=-\infty$ as well, and from \eqref{2.13} we get $w=-\infty$,
	which is a contradiction. Thus $c=0$, and the proof is complete.
\end{proof}

\begin{remark}\label {rmk2.1} \rm{ It is easy to verify from the above proof  that the assumption (H3)(ii) is not needed at all if (H3)(i) holds with $\lim_n\Big(\prod_{k=1}^na_k\Big)=0$.}
\end{remark}

The goal now is to show that (H1), (H4), (H5) are sufficient conditions for the trivial equilibrium to be a global attractor of the oscillatory solutions of \eqref{1.1)}. 
A first auxiliary result is crucial to establish upper and lower bounds for  oscillatory solutions, and was
inspired in arguments of   \cite{19.}. 
%
%
%
%
\begin{lemma}\label{lem2.4}  
	Assume (H1), (H4),  and 
	$$\alpha_1\alpha_2\le 1
	$$
	for some $\al_1=\al_1(T),\al_2=\al_2(T)$ as in \eqref{2.3}. Let  $y(t)$ be a solution of \eqref{2.9} on $[0,\infty)$ and  $t_0\ge T$ such that $y(t_0)=0$. 
	Then, for any $\eta>0$, the following conditions hold:\vskip 0.1cm
	(i) If $-\eta\le y(t)\le \eta \alpha_2$ for $t\in[t_0-\tau(t_0),t_0]$, then $-\eta\le y(t)\leq\eta \alpha_2$ for all $t>t_0$;
	\vskip 0.1cm
	(ii) If $-\eta \alpha_1\le y(t)\le \eta$ for $t\in[t_0-\tau(t_0),t_0]$, then $-\eta \alpha_1\leq y(t)\le \eta$ for all $t>t_0$.
\end{lemma}

\begin{proof}  We shall prove (i), the proof of (ii) being similar.
	If the assertion (i) is false,  there exists $T_0>t_0$ such that either $y(T_0)> \al_2 \eta$ or $y(T_0)<-\eta$. We consider these two situations separately. \smal
	
	{\it Case 1}. Suppose that $y(T_0)>\al_2 \eta$ for some $T_0>t_0$, with $-\eta\le y(t)<y(T_0)$ for $t\in [t_0,T_0)$.  \smal
	
	We first prove that there is $\xi_0\in [T_0-\tau(T_0),T_0]$ such that $y(\xi_0)=0$.  Otherwise, we obtain necessarily that $y(t)>0$ for $t\in [T_0-\de-\tau(T_0-\de),T_0]$ and some small $\de>0$ (recall that $y(t)$ is continuous), and from \eqref{2.9} and (H4) it follows that
	$$y'(t)\le -a(t)y(t)+ \prod_{k:0\leq t_k<t}J_k(x(t_k))\, 
	\lambda_2(t) {\mathcal M}_t(-x_t)
	\leq 0,\q t\in [T_0-\de,T_0],$$
	implying that $y(T_0-\delta)\geq y(T_0)$, which contradicts the definition of $T_0$.
	
	Choose $\xi_0\in [T_0-\tau(T_0),T_0]$ such that $y(\xi_0)=0$. We may suppose that $y(t)>0$ for $\xi_0<t<T_0$, thus $t_0\le \xi_0$.
	By \eqref{2.2}, \eqref{2.8}, \eqref{2.9}, and \eqref{2.10}, for $s\in[\xi_0,T_0]\setminus\{t_k\}$ we obtain
	\begin{equation}\label{2.14}
	\begin{split}
	\left(e^{A(s)}y(s)\right)'&=\prod_{k:0\leq t_k<s}J_k(x(t_k))e^{A(s)}f(s,x_s)\\
	&\leq  e^{A(s)}\lambda_2(s)\prod_{k:0\leq t_k<s}J_k(x(t_k)){\mathcal M}_s(-x_s)\\
	&=e^{A(s)}\lambda_2(s)\prod_{k:0\leq t_k<s}\!\!J_k(x(t_k))\cdot\\
	&\,\,\,\,\cdot\max\left\{0,\sup_{\theta\in[-\tau(s),0]}\left(-y(s+\theta)\prod_{k:0\leq t_k<s+\theta}J_k(x(t_k))^{-1}\right)\right\}\\
	&=e^{A(s)}\lambda_2(s) \max\left\{0,\sup_{\theta\in[-\tau(s),0]}\left(-y(s+\theta)\prod_{k:s+\theta\leq t_k<s}J_k(x(t_k))\right)\right\} \\
	&\leq  e^{A(s)}\lambda_2(s)B(s){\mathcal M}_s(-y_s),
	\end{split}
	\end{equation}	
	with $B(s)$ defined as in (H5).
	Now,  for $s\in [\xi_0,T_0], s\ne t_k$, we have $y_s(\th)\ge -\eta$ for $s\in [\xi_0,T_0],\th \in [s-\tau(s),s]$, thus 
	$$
	\left(e^{A(s)}y(s)\right)'\leq  \eta e^{A(s)}\lambda_2(s)B(s) .$$
	Integrating over $[\xi_0,T_0]$, we get
	\begin{equation*}
	\begin{split}
	y(T_0)&\le   \eta e^{-A(T_0)}\int_{\xi_0}^{T_0} e^{A(s)}\la_2(s) B(s)\, ds\\
	&  =  \eta \int_{\xi_0}^{T_0} e^{-\int_s^{T_0} a(u)\, du} \la_2(s) B(s)\, ds\le \al_2\eta,
	\end{split}
	\end{equation*}   
	which contradicts the definition of $T_0$.
	
	\med
	
	{\it Case 2}. Suppose that $y(T_0)<- \eta$ for some $T_0>t_0$, with $y(T_0)< y(t)\le \al_2\eta$ for $t\in [t_0,T_0)$.  \smal
	
	Reasoning as above, we deduce  that there is $\xi_0\in [t_0,T_0)\cap [T_0-\tau(T_0),T_0)$ such that $y(\xi_0)=0$ and $y(t)<0$ for  $\xi_0<t<T_0$. Since $y_t(\th)\le \al_2\eta$ for $t\in [\xi_0,T_0],t\ne t_k$, $\th \in [t-\tau(t),t]$, we now obtain
	$$
	\left(e^{A(t)}y(t)\right)'\ge -e^{A(t)}\lambda_1(t)B(t) \al_2\eta .$$
	Integrating over $[\xi_0,T_0]$, and using the inequality  $\al_1\al_2\le 1$,
	we get
	\begin{equation*}
	\begin{split}
	y(T_0)\ge &  - \al_2\eta e^{-A(T_0)}\int_{\xi_0}^{T_0} e^{A(s)}\la_1(s) B(s)\, ds\\
	= & - \al_2\eta  \int_{\xi_0}^{T_0} e^{-\int_s^{T_0} a(u)\, du} \la_1(s) B(s)\, ds\ge -\al_1\al_2\eta\ge -\eta,
	\end{split}
	\end{equation*}      
	which is not possible.
\end{proof}

Sufficient conditions for the boundedness of all solutions of \eqref{1.1)}
follow immediately from Lemmas  \ref{lem2.2}, \ref{lem2.3} and \ref{lem2.4}.

\begin{theorem}\label{thm2.1}
	Assume (H1), (H2)(i), (H4)   and $\alpha_1\alpha_2\le 1$, where $\al_1,\al_2$ are as in \eqref{2.3}. Then all solutions of \eqref{1.1)} are bounded. Furthermore, if in addition either (H2)(ii) or (H3) are satisfied, the zero solution of \eqref{1.1)} is (uniformly) stable.
\end{theorem}

We are now in the position to prove the main result in this section.

\begin{theorem}\label{thm2.2} Assume (H1), either (H2) or (H3), (H4) and (H5).  Then the zero solution of \eqref{1.1)} is globally asymptotically stable.
\end{theorem}

\begin{proof} By virtue of Lemma \ref{lem2.3} and \eqref{2.11}, we only need to show that zero attracts all  oscillatory solutions of \eqref{2.9}.
	Take an oscillatory solution $y(t)$ of \eqref{2.9}, and set
	\begin{equation}\label{2.15} 
	u:=\limsup_{t\to\infty}y(t),\q -v:=\liminf_{t\to\infty}y(t).
	\end{equation} 
	Clearly $0\leq u,v<\infty$, because $y(t)$ is bounded. 
	For $t\ge 0$, denote $d(t)=t-\tau(t)$ and $d^2(t)=d(d(t))$. 
	Fix $\varepsilon>0$ and consider $T_0>0$ as in (H5) with $ d^2(T_0)>0$ and such that
	\begin{equation}\label{2.16} 
	-v_\varepsilon:=-(v+\varepsilon)<y(t)<u+\varepsilon:=u_\varepsilon,\q{\rm for }\q t\geq d^2(T_0).
	\end{equation} 
	As $y(t)$ is continuous, there exists a sequence  $s_n\nearrow\infty$ with $s_n\geq T_0$ such that $y(s_n)>0$, $y(s_n)$ are local maxima, and $y(s_n)\to u$ as $n\to\infty$. We may assume that $y(s)<y(s_n)$ for $s_n-s>0$ small. As in the proof of Lemma \ref{lem2.4}, by the Yorke condition (H4) we deduce that for each $n\in\N$  there exists $\xi_n\in[s_n-\tau(s_n),s_n)$ such that $y(\xi_n)=0$ and $y(s)>0$ for $s\in(\xi_n,s_n]$. By \eqref{2.16}, we have $y(s)>-v_\varepsilon$ for $s\in[\xi_n-\tau(\xi_n),s_n]$. Arguing as in the proof of Case 1  of Lemma \ref{lem2.4}(i), we conclude that $y(s_n)\leq \alpha_2  v_\vare$. Letting $n\to\infty$ and $\vare\to0^+$, we obtain
	\begin{equation}\label{2.17} 
	u\leq \alpha_2 v.\end{equation} 
	Similar arguments lead to
	\begin{equation}\label{2.18} 
	v\leq \alpha_1u.\end{equation} 
	From \eqref{2.17}, \eqref{2.18}, we have
	\begin{equation}\label{2.19} 
	u\leq \alpha_1\alpha_2u,\q v\leq \alpha_1\alpha_2 v.\end{equation} 
	Under the constraint $\al_1\al_2<1$, \eqref{2.19} is possible only if  $u=0$ and $v=0$, thus  $y(t)\to 0$ as $t\to \infty$.
\end{proof}

%
%
%
%

For the situation without impulses we obtain the following criterion:

\begin{corollary}\label{cor2.1} For $a(t), \tau(t),f(t,\var)$ as in \eqref{1.1)}, consider the scalar DDE
	\begin{equation}\label{2.20} 
	x'(t)+a(t)x(t)=f(t, x_t),\q t\ge 0,\end{equation} 
	and assume  either (H2)(ii) or (H3)(ii), (H4) and $\al_1\al_2<1$, where
	\begin{equation}\label{2.21} 
	\al_i:=\al_i(T)=\sup_{t\geq T}
	\int_{t-\tau(t)}^t\lambda_i(s)e^{-\int_s^t a(u)du}ds,\q i=1,2,\end{equation} 
	for some $T>0$. 
	Then the zero solution of \eqref{2.20} is globally asymptotically stable.
\end{corollary}

Even for the situation without impulses \eqref{2.20}, when $a(t)\not\equiv0$ but $\int_0^\infty a(t)\, dt<\infty$, the additional requirement (H3)(ii) must be imposed together with   (H4),  otherwise zero need not attract the eventually monotone solutions, as shown by the next counter-example.

\begin{exmp} \rm{Consider the scalar DDE without impulses
		\begin{equation}\label{2.22} 
		x'(t)+{1\over {(t+1)^2}} x(t)=g(x(t-\tau)),\q t\ge 0,\end{equation} 
		where $\tau >0$ and $g:\R\to\R$ is defined by $g(x)=-x$ for $x\le 0$, $g(x)=0$ for $x>0$. It is apparent that  
		\eqref{2.2} is satisfied with $\la_1(t)\equiv 0,\la_2(t)\equiv 1$. 
		For $\al_i=\al_i(T), i=1,2,$ defined in \eqref{2.21}, we get $\al_1=0$ and  $\al_2<  \tau<\infty$.
		Therefore, (H5) holds and  zero is an attractor of  all oscillatory solutions. Note however that both (H2)(ii) and (H3)(ii) fail, due to 
		$$\int_0^\infty {1\over {(t+1)^2}}\, dt=1<\infty\q {\rm and}\q 
		\int_{0}^\infty |g(x(t-\tau))|\, dt <\infty$$
		if $\lim_{t\to\infty} x(t)=c>0$. On the other hand, observe that   zero is not globally asymptotically stable for \eqref{2.22}, since
		all  functions $x(t)=c\exp ({1\over{t+1}})$ with $c>0$ are solutions of \eqref{2.22}.}
\end{exmp}

Theorem \ref{thm2.2} can be slightly improved for systems \eqref{1.1)} with $f(t,x_t)$ of the form
\begin{equation}\label{2.23} 
f(t,x_t)=\sum_{i=1}^n f_i(t,x_t^i),\end{equation} 
where $f_i(t,x_t^i)=f_i(t,x_{|_{[t-\tau_i(t),t]}})$ with $\tau_i(t)$ satisfying the above conditions for $\tau(t)$, as follows.  Let  $\tau(t)=\max_{1\le i\le n} \tau_i(t)$.  
Suppose now that  each $f_i,\, i=1,\dots,n$, satisfies the  Yorke condition (H4),
i.e., there exist piecewise continuous functions
$\lambda_{1,i},\lambda_{2,i}:[0,\infty)\to[0,\infty)$
such that
\begin{equation}\label{2.24} 
-\lambda_{1,i}(t){\mathcal M}^i_t(\varphi)\leq f_i(t,\varphi_{|_{[-\tau_i(t),0]}})\leq\lambda_{2,i}(t){\mathcal M}^i_t(-\varphi),\q t\geq0,1\le i
\le n,\q \varphi\in PC^i(t),\end{equation} 
for $ PC^i(t)=PC([-\tau_i(t),0];\R)$ and ${\mathcal M}^i_t(\varphi)=\dps \max\{ 0,\sup _{\th\in [-\tau_i(t),0]}\var (\th)\}$. A careful reading of the proof of Lemma \ref{lem2.4} (see formula \eqref{2.14})  shows the validity of the statement below.

\begin{theorem}\label{thm2.3} For \eqref{1.1)} with $f(t,x_t)$ of the form \eqref{2.23}, assume (H1) and  either (H2) or (H3). Suppose also that the Yorke conditions \eqref{2.24} hold and that there is $T>0$ such that
	$\al_1\al_2<1$, where $  \alpha_j:=  \alpha_j(T)$ are given by
	\begin{equation}\label{2.25} 
	\alpha_j(T):=\sup_{t\geq T} 
	\int_{t-\tau(t)}^t \sum_{i=1}^n\lambda_{j,i}(s)e^{-\int_s^ta(u)du}B_i(s)\, ds,\q j=1,2,\end{equation} 
	and    $B_i(t):=\dps\max_{\theta\in[-\tau_i(t),0]}\bigg(\prod_{k:t+\theta\leq t_k<t}b_k^{-1}\bigg),\, i=1,\dots,n$.
	Then the zero solution of \eqref{1.1)} is globally asymptotically stable.
\end{theorem}

The next criterion is an important particular case of Theorem \ref{thm2.3}.

\begin{corollary}\label{cor2.2} Consider \eqref{1.1)} with $f(t,x_t)$ given by \eqref{2.23}.
	Assume (H1), either (H2) or (H3), and that
	conditions \eqref{2.24} hold with 
	$$\sum_{i=1}^n\la_{j,i}(t)B_i(t)\le c_ja(t),\q j=1,2,$$ where $c_1,c_2$ are positive constants.
	If either 
	\begin{equation}\label{2.26} A:=\limsup_{t\ge 0} \int_{t-\tau(t)}^t a(u)\, du<\infty\end{equation}
	with
	$\sqrt{c_1c_2}(1-e^{-A})<1,$
	or $A=\infty$ with $c_1c_2<1$,
	then the zero solution of \eqref{1.1)} is globally asymptotically stable. In particular this is the case if  conditions \eqref{2.24} and \eqref{2.26}  are satisfied with
	$$\sum_{i=1}^n\la_{j,i}(t)B_i(t)\le a(t),\q j=1,2.$$
\end{corollary}

\begin{proof} Suppose $A<\infty$.  In this scenario,  we may take $\al_j$ in \eqref{2.25}
	given by
	\begin{equation}\label{2.27} 
	\al_j=\al_j (T)=
	c_j\sup_{t\geq T}
	\int_{t-\tau(t)}^t a(s)e^{-\int_s^ta(u)du}\, ds=c_j\sup_{t\geq T}\left (1-e^{-\int_{t-\tau(t)}^t a(u)\, du}\right),
	\end{equation}    
	therefore, for any $\vare>0$  there is $T>0$ such that   $\al_j\le c_j(1-e^{-(A+\vare)}), \, j=1,2$. If  $A=\infty$, instead of \eqref{2.27} we obtain $\al_j=c_j,j=1,2.$\end{proof}

As a particular case,  a criterion given by Tang  \cite{9.} and stated below is obtained by considering the DDE without impulses \eqref{2.20}
and taking $c_1=c_2=1$ in  Corollary \ref{cor2.2}.

\begin{corollary}\label{cor2.3} For $a(t), \tau(t),f(t,\var)$ as in \eqref{1.1)}, 
	assume  (H2)(ii), \eqref{2.26} and
	$$
	-a(t){\mathcal M}_t(\varphi)\leq f(t,\varphi)\leq a(t){\mathcal M}_t(-\varphi),\q t\geq0,\q \varphi\in C(t),
	$$
	where ${\mathcal M}_t(\varphi)$ is as in \eqref{2.2}.
	Then the zero solution of \eqref{2.20} is globally asymptotically stable.
\end{corollary}

\section{A periodic Lasota-Wazewska model with impulses}\label{seccao-modelo-LasotaWazewska}
\setcounter{equation}{0}

In this section, we study a periodic Lasota-Wazewska model with impulses (see  e.g. \cite{5.,  6., 12.}):
\begin{equation}\label{3.1}
\begin{split}
&N'(t)+a(t)N(t)=\sum_{i=1}^n b_i(t)e^{-\be_i (t)N(t-\tau_i(t))},\q 0\le t\ne t_k, \\
&\Delta N(t_k):=N(t_k^+)-N(t_k) =I_k(N(t_k)),\q k=1,2,\dots,
\end{split}
\end{equation}
where $0<t_1<t_2<\dots <t_k<\dots, t_k\to \infty$, and

\begin{itemize}
	
	\item[($f_0$)]
	the functions $a(t),b_i(t),\be_i(t),\tau_i(t)$ are continuous, positive and $\omega$-periodic, $1\le i\le n, t\in\R$, for some constant $\omega >0$;

	\item[($i_0$)]   the functions $I_k:[0,\infty)\to \R$ are  continuous with  $u+I_k(u)>0$ for $u>0$, $ k\in\N$; moreover, there is a positive integer $p$ such that 
	$$
	t_{k+p}=t_k+\omega,\q I_{k+p}(u)=I_k(u),\q  k\in\N, u>0.
	$$
\end{itemize}
Special attention will be given to the particular case of \eqref{3.1} with $\omega$-periodic constant delays $\tau_i(t)=m_i\omega$ for $m_i$ positive integers, $ i=1,\dots,n$. For some related models, see also \cite{HuoLiLui,SakerAl1,SakerAl2}.

Without loss of generality, we may suppose that  there are exactly $p$ instants of impulses on the interval $[0,\omega]$, $t_1,t_2,\dots, t_p$. With minimal changes, we can also consider a more general framework, with $a(t),b_i(t),\be_i(t),\tau_i(t)$  piecewise continuous functions.  Due to the biological applications, we only consider positive solutions  of \eqref{3.1}, corresponding to initial conditions $N_0=\phi\in C^+_0$, where $C^+_0$ is the set of continuous functions $ \phi:[-\bar\tau,0]\to [0,\infty)$ with $\phi(0)>0$, for $\dps \bar \tau=\max_{1\le i\le n}\max_{t\ge 0} \tau_i(t)$. It is easy to see that if $\phi\in C^+_0$, then $N(t,0,\phi)$ is defined and positive for $t\ge 0$.

Some criteria for the existence of an $\omega$-periodic solution to \eqref{3.1} have been established. Namely, the following result is a consequence of Theorem 2.3  in Li et al. \cite{5.}:

\begin{lemma}\label{lem3.1} \cite{5.} For \eqref{3.1}, assume ($f_0$), ($i_0$), and that $I_k(u)\ge 0$ for $u\ge 0, k\in\N$. If
	$$CI^\infty<1,$$
	where $\dps I^\infty=\limsup_{u\to \infty} \sum_{k=1}^p {{I_k(u)}\over u}$ and $\dps C={{e^{\int_0^\omega a(u)\, du}}\over {e^{\int_0^\omega a(u)\, du}-1}},$ then \eqref{3.1} has at least one $\omega$-periodic positive solution $N^*(t)$.
\end{lemma}

Let us however mention that  the assumption $I_k(u)\ge 0$ for all $u\ge 0,k\in\N$ in  \cite{5.} is quite strong, since it requires that the impulses are always positive. In view of the biological meaning of the model, the natural constraint is only that $I_k(u)+u>0$ for $u>0$. On the other hand, as we shall see, Liu and Takeuchi  \cite{6.}  considered the particular case of \eqref{3.1} with $\omega$-periodic constant delays $\tau_i(t)=m_i\omega$ for $m_i$ positive integers, $ i=1,\dots,n,$ and linear  impulses  $I_k(u)=b_ku$ with constants $b_k>-1$ (see  system \eqref{3.16} addressed later in this section), for which  the existence of a positive $\omega$-periodic solution was proven under a very general condition. 

For system \eqref{3.1} with   ($f_0$),  ($i_0$) fulfilled, we now 
impose the following additional hypotheses:

\begin{itemize}
	
	\item[($i_1$)]there exist constants $a_1,\dots,a_p$ and $b_1,\dots,b_p$, with $b_k>-1$, and such that
	$$b_k\le {{I_k(x)-I_k(y)}\over{x-y}}\le a_k,\q x,y\ge 0, x\ne y, k=1,\dots,p;$$
	\item[($i_2$)] $\dps \prod_{k=1}^p (1+a_k)\le 1.$
\end{itemize}

Assume now that there exists a positive $\omega$-periodic  solution $N^*(t)$, and
effect the change of variables $\dps x(t)=N(t)-N^*(t)$. Eq. \eqref{3.1} is transformed into
\begin{equation}\label{3.2}
\begin{split}
&x'(t)+a(t)x(t)=f(t,x_t),\q 0\le t\ne t_k, \\
&\Delta x(t_k) =\tilde I_k(x(t_k)),\q  k\in\N,
\end{split}
\end{equation}
where 
\begin{equation}\label{3.3}
\begin{split}
&f(t,\var)=\sum_{i=1}^n b_i(t)e^{-c_i(t)}\Big [e^{-\be_i (t)\var(-\tau_i(t))}-1\Big],\\
&c_i(t)=\be_i(t)N^*(t-\tau_i(t)),\q i=1,\dots,n,\\
&\tilde I_k(u)=I_k\big (N^*(t_k)+u\big )-I_k\big (N^*(t_k)\big ), \q k=1,\dots,p.
\end{split}
\end{equation}

Now, we insert this transformed system into the  framework of the previous sections: \eqref{3.2}
has the form \eqref{1.1)}, where the function $f(t,x_t)$ may have jump discontinuities  at the  points $t$ such that  $t-\tau_i(t)=t_k$ for some  $1\le i\le n,\, k\in\N$. Recall that $x(t)+N^*(t)>0$ for $t\ge 0$, for any solution $x(t)$ of \eqref{3.2}. Naturally, 
$S=\{\var \in PC([-\bar\tau,0];\R): \var (\th)\ge -N^*(\th)\, {\rm for}\, -\bar\tau\le \th <0, \var (0)>-N^*(0)\}$
is taken as the set of admissible initial conditions for \eqref{3.2}, and the spaces $PC^i(t)$ in \eqref{2.24} are replaced by
$S^i(t)=\{\var \in PC^i(t): \var (\th)\ge -N^*(t-\th)\ {\rm for} \ -\tau_i(t)\le \th \le 0\}$.

\begin{lemma}\label{lem3.2}
	Under ($f_0$), ($i_0$)--($i_2$), system \eqref{3.2} satisfies (H1) and (H2).\end{lemma}

\begin{proof} Let $f$ and $\tilde I_k$ be defined by \eqref{3.3}.
	It is  apparent that (H1) and (H2)(i) hold for \eqref{3.2}, since
	$$
	\tilde b_k\le {{u+\tilde I_k(u)}\over u}\le \tilde a_k,\q k\in\N, u\ne 0,u>-N^*(t_k),
	$$
	with
	$$\tilde b_k=b_k+1>0,\q  \tilde a_k= a_k+1,
	$$
	and the positive sequence $(P_k)_{k\in\N}$ defined by $\dps P_k=\prod_{i=1}^k \tilde a_i$ is  bounded, with
	$$\dps P_k\le \max \{ \prod_{i=1}^j \tilde a_i: j=1,2,\dots,p\}.$$
	Since $a(t)$ is positive and $\omega$-periodic, we have
	$\int_0^\infty a(t)\, dt=\infty$.
\end{proof}

We now choose adequate $\la_1,\la_2$ for the Yorke condition to be satisfied.

\begin{lemma}\label{lem3.3}
	Assume ($f_0$), ($i_0$)--($i_2$). If $N^*(t)$ is a positive $\omega$-periodic solution of system \eqref{3.1}, then \eqref{3.2} satisfies   \eqref{2.24}  with $\ \la_{1,i}, \la_{2,i}:[0,\infty)\to [0,\infty)$ given by
	\begin{equation}\label{3.4}
	\la_{1,i}(t)=\be_i(t)b_i(t)e^{-\be_i(t)N^*(t-\tau_i(t))},\q
	\la_{2,i}(t)=\be_i(t)b_i(t),\q t\ge 0,1\le i\le n.\end{equation}
\end{lemma}

\begin{proof} Write $c_i(t), f_i(t,\var)$ as in \eqref{3.3}. 
	Take $t\ge 0, i\in \{1,\dots, n\}$ and $\var\in S^i(t)$. We have
	$$e^{-\be_i(t)\var(-\tau_i(t))}-1\ge e^{-\be_i(t){\mathcal M}^i_t(\var)}-1\ge -\be_i(t){\mathcal M}^i_t(\var),$$
	thus
	$f_i(t,\var)\ge -\be_i(t)b_i(t)e^{-c_i(t)}{\mathcal M}^i_t(\var).$
	On the other hand, by Lagrange's intermediate value theorem, 
	$f_i(t,\var)=-b_i(t)\be_i(t)\var(-\tau_i(t)) e^{-d_i(t)}$ for some $d_i(t)$ between $c_i(t)$ and $c_i(t)+\be_i (t)\var(-\tau_i(t))>0$, hence
	$f_i(t,\var)\le \be_i(t)b_i(t){\mathcal M}^i_t(-\var).$
\end{proof}

The next criterion follows as an immediate consequence of  Theorem \ref{thm2.3}.

\begin{theorem}\label{thm3.1}
	Assume ($f_0$), ($i_0$)--($i_2$), and that there is a  positive $\omega$-periodic solution $N^*(t)$  of system \eqref{3.1}.  If there is  $T>\bar\tau$ such that $\al_1\al_2<1$ with
	\begin{equation*}  
	\begin{split}
	\al_1:&=\sup_{t\ge T}\int_{t-\tau(t)}^t \sum_{i=1}^n \be_i(s)b_i(s)e^{-\be_i(s)N^*(s-\tau_i(s))}B_i(s)e^{-\int_s^t a(u)\, du}\, ds\\
	\al_2:&=\sup_{t\ge T}\int_{t-\tau(t)}^t \sum_{i=1}^n \be_i(s)b_i(s)B_i(s)e^{-\int_s^t a(u)\, du}\, ds,\end{split}
	\end{equation*}
	where  $\dps B_i(t)=\max_{\th\in [-\tau_i(t),0]}\bigg (\prod_{k: t+\th\le t_k<t}(1+b_k)^{-1}\bigg)$, then  $N^*(t)$ is globally asymptotically stable, in the sense that it is stable and any positive solution $N(t)$ of \eqref{3.1} satisfies
	$\lim_{t\to\infty} \big (N(t)-N^*(t)\big )=0.$
\end{theorem}

Since $a(t),\tau(t)$ are uniformly bounded, condition \eqref{2.26} is trivially satisfied.
Taking
$\la_{j,i}(s)= \be_i(t) b_i(t),\, j=1,2,i=1,\dots, n$,  Corollary \ref{cor2.2} provides an alternative criterion:

\begin{theorem}\label{thm3.2}
	Assume ($f_0$), ($i_0$)--($i_2$), and that there is a  positive $\omega$-periodic solution $N^*(t)$  of system \eqref{3.1}.  Then  $N^*(t)$ is globally attractive if
	\begin{equation}\label{3.5}
	\sum_{i=1}^n \be_i(t)b_i(t)\max_{\th\in[-\tau_i(t),0]}\prod_{k:t_k\in [t+\th,t)} \!\!(1+b_k)^{-1}\le a(t),\q t\ge 0.
	\end{equation}
\end{theorem}

In what follows, we use the notations
$ \be(t)=\max_{1\le i\le n}\be_i(t)$ 
and, for an $\omega$-periodic function $g$,
$$\overline g=\max_{0\le t\le \omega}g(t).$$
Namely,  $\overline{N^*}=\dps\max_{0\le t\le\omega} N^*(t)$ and ${\overline{\be N^*}}=\dps\max_{0\le t\le\omega} \max_{1\le i\le n}\be_i(t)N^*(t)$.

\med

\begin{remark}\label{rmk3.1} \rm{ For Theorem \ref{thm3.2}, we used    conditions \eqref{2.24} with
		$\la_{1,i}(t)=\la_{2,i}(t)= \be_i(t) b_i(t), 1\le i\le n
		$;
		the choice of $\la_{1,i}(t)$ in \eqref{3.4} however provides  a better estimate. Alternatively, instead of $\la_{2,i}(t)= b_i(t)\be_i(t)$,
		one can choose
		\begin{equation}\label{3.6}
		\la_{2,i}(t)= {1\over{\overline{N^*}}}\, (e^{\overline{\be N^*}}-1)\, b_i(t)e^{-\be_i(t)N^*(t-\tau_i(t))}\,, \ t\ge 0,1\le i\le n.
		\end{equation}
		In fact, for $t\ge 0, i\in \{1,\dots, n\}$ and $\var\in S^i(t)$, we obtain the inequality
		$$e^{-\be_i(t)\var(-\tau_i(t))}-1\le e^{\be_i(t){\mathcal M}^i_t(-\var)}-1\le  {{e^{\overline{\be N^*}}-1}\over{\overline{N^*}}}{\mathcal M}_t^i(-\var),$$
		which yields
		$f_i(t,\var)\le \la_{2,i}(t){\mathcal M}^i_t(-\var)$  for $\la_{2,i}(t)$ as in \eqref{3.6}.
		Although the choice of $ \la_{2,i}(t)$ in \eqref{3.4}  is always better than the one above (because
		with $x=\overline{\be N^*}$ and $c_i(t)= \be_i(t)N^*(t-\tau_i(t))\le x$, we have  $(e^x-1)e^{-c_i(t)}-x\ge 0$ for $x\ge 0$),  the use of $ \la_{2,i}(t)$ as in \eqref{3.6}   turns out to be  effective for the special case of \eqref{3.1} analysed below.}
\end{remark}

We now proceed with a deeper analysis of the case of \eqref{3.1} with time independent delays multiple of the period,
$$\tau_i(t)\equiv m_i\omega,$$
where $m_i$ are positive integers, $i=1,\dots,n$. In this situation, \eqref{3.2} becomes
\begin{equation}\label{3.7}
\begin{split}
&x'(t)+a(t)x(t)=\sum_{i=1}^n b_i(t)e^{-\be_i(t)N^*(t)}\Big [e^{-\be_i (t)x(t-m_i\omega)}-1\Big],\q  0\le t\ne t_k, \\
&\Delta x(t_k) =\tilde I_k(x(t_k)),\q  k\in\N,
\end{split}
\end{equation}
where 
$\tilde I_k(u)=I_k\big (N^*(t_k)+u\big )-I_k\big (N^*(t_k)\big )$ are as in \eqref{3.3}.

%

In what follows  we establish  upper estimates for  $\al_1,\al_2$ defined by \eqref{2.25} which, although  not as sharp as the ones in Theorem \ref{thm3.1},  are easier to handle. To simplify the exposition, assume $I_k(0)=0$ for $ k\in \N$ -- which is natural from a biological point of view --, so that $I_k(N^*(t_k))\ge b_k N^*(t_k)$; otherwise, $I_k(N^*(t_k))\ge{I_k}(0)+b_kN^*(t_k), 1\le k\le p$,
and straightforward changes should be introduced in the computations below. 

\begin{theorem}\label{thm3.3}
	Consider \eqref{3.1} with $\tau_i(t)\equiv m_i\omega$ ($m_i\in\N$) and denote $\bar{m}=\max_{1\le i\le n} m_i$. Assume ($f_0$), ($i_0$)--($i_2$) with $I_k(0)=0$ for $k\in\N$, and that there is a  positive $\omega$-periodic solution $N^*(t)$  of system \eqref{3.1}.  If  
	\begin{equation}\label{3.8}
	\begin{split}
	B^{\bar{m}} \left(\overline {\be} \overline{ N^*}(e^{\overline {\be N^*}}-1)\right)^{1\over 2} 
	\Big (1-& e^{-\bar{m} \int_0^{\omega} a(u)\, du}\Big)\cdot\\
	&\cdot\left [1-\Big (1-e^{-\int_0^\omega a(u)\, du}\Big)^{-1}\sum_{k=1}^p \min (b_k,0) \right]<1,
	\end{split}
	\end{equation}
	where $\dps B=\max_{1\le l,j\le p}\prod_{k=1}^j(1+b_{l+k})^{-1}$,  then   $N^*(t)$ attracts any positive solution $N(t)$ of \eqref{3.1}. 
\end{theorem}
\begin{proof} Condition ($i_2$) implies  $B\ge 1$  for $B$ defined above, hence    for $t\ge 0$ and $1\le i\le n$
	$$B_i(t):=\max_{\th\in [-m_i\omega,0]}\left (\prod_{k:t+\th\le t_k<t}(b_k+1)^{-1}\right)\le B^{m_i}\le B^{\bar{m}}.$$
	For the sake of simplicity, in what follows we  suppose that  the coefficients $b_k$ in  $(i_1)$ satisfy $b_k\in (-1,0]\, (1\le k\le p)$; otherwise we may replace $b_k$ by $\min\{ 0,b_k\}$.

	Now, choose  $\la_{1,i}$ as in \eqref{3.4} and $\la_{2,i}$  as in \eqref{3.6}. For \eqref{3.7} we obtain
	$$\la_{1,i}(t)= \be_i(t)b_i(t)e^{-\be_i(t)N^*(t)},\ \la_{2,i}(t)=  {1\over{\overline{N^*}}}\, (e^{\overline{\be N^*}}-1)  b_i(t)e^{-\be_i(t)N^*(t)},\,  0\le t\ne t_k.$$
	Since $N^*(t)$ is an $\omega$-periodic solution of  \eqref{3.1},  for $t>0, t\ne t_k,$
	\begin{equation}\label{3.9}
	\sum_{i=1}^n b_i(t) e^{-\be_i(t)N^*(t)}=(N^*)'(t)+a(t)N^*(t),
	\end{equation}
	with $N^*(t)$ having possible jumps at the points $t_k$. 
	From \eqref{3.9}, for $t>0$ we  derive
	\begin{equation}\label{3.10}
	\begin{split}
	\al_1(t)&:=\int_{t-\bar{m}\omega}^t \sum_{i=1}^n \la_{1,i}(s)B_i(s)e^{-\int_s^t a(u)\, du}\,  ds\\
	&\le \bar \be B^{\bar{m}} \int_{t-\bar{m}\omega}^t {d\over {ds}}\Big [N^*(s)e^{-\int_s^t a(u)\, du}\Big ]\, ds\\
	&= \bar \be B^{\bar{m}}  \left [N^*(t)\Big (1- e^{-\bar{m} \int_0^{\omega} a(u)\, du}\Big)-\sum_{k:t_k\in [t-\bar{m}\omega,t)} I_k(N^*(t_k))e^{-\int_{t_k}^t a(u)\, du}\right] \\
	&\le \bar \be B^{\bar{m}}  \left [N^*(t)\Big (1- e^{-\bar{m} \int_0^{\omega} a(u)\, du}\Big)-\sum_{k:t_k\in [t-\bar{m}\omega,t)} b_kN^*(t_k)e^{-\int_{t_k}^t a(u)\, du}\right] .
	\end{split}
	\end{equation}
	With $z_k=b_kN^*(t_k)$, we have
	\begin{equation}\label{3.11}
	\begin{split}
	\sum_{k:t_k\in [t-\bar{m}\omega,t)} z_ke^{-\int_{t_k}^t a(u)\, du}
	& =\left (\sum _{k:t_k\in [t-\omega,t)} z_ke^{-\int_{t_k}^t a(u)\, du}\right)\cdot\\
	&\,\,\,\,\,\cdot \left (1+e^{-\int_0^\omega a(u)\, du}+\cdots+
	e^{-(\bar{m}-1)\int_0^\omega a(u)\, du}\right)\\
	&\ge \left (\sum _{k=1}^p z_k\right)
	{{1-e^{-\bar{m}\int_0^\omega a(u)\, du}}\over {1-e^{-\int_0^\omega a(u)\, du}}}.
	\end{split}
	\end{equation}
	The estimates \eqref{3.10} and \eqref{3.11} yield
	\begin{equation*}
	\al_1(t)
	\le  B^{\bar{m}} \bar \be  \overline{N^*}\Big (1- e^{-\bar{m} \int_0^{\omega} a(u)\, du}\Big)
	\left [1-\Big (1-e^{-\int_0^\omega a(u)\, du}\Big)^{-1}\sum_{k=1}^p b_k\right]=:\sigma_1.\end{equation*}
	In a similar way, we obtain
	\begin{equation*}
	\begin{split}
	\al_2(t)&:=\int_{t-\bar{m}\omega}^t \sum_{i=1}^n \la_{2,i}(s)B_i(s)e^{-\int_s^t a(u)\, du}\, ds\\
	&\le  B^{\bar{m}}(e^{\overline {\be N^*}}-1)\Big (1- e^{-\bar{m} \int_0^{\omega} a(u)\, du}\Big) \left [1-\Big (1-e^{-\int_0^\omega a(u)\, du}\Big)^{-1}\sum_{k=1}^p b_k\right]=: \sigma_2.
	\end{split}
	\end{equation*}
	Clearly, condition $\sigma_1\sigma_2<1$ is equivalent to \eqref{3.8}.
\end{proof}

As  a by-product, we obtain some results  for DDEs without impulses by setting $b_k=a_k=0$ for $1\le k\le p$ in the above theorems.

\begin{corollary}\label{cor3.1}   Consider 
	\begin{equation}\label{3.14}
	N'(t)+a(t)N(t)=\sum_{i=1}^n b_i(t)e^{-\be_i (t)N(t-m_i\omega)},\q t\ge 0,
	\end{equation}
	where $\omega>0, m_i\in\N$ and the coefficient functions satisfy ($f_0$). Let $\bar{m}= \max _{1\le i\le n} m_i$.  Then, there is a  positive $\omega$-periodic solution $N^*(t)$, which is a global attractor of all positive solutions if one of the following conditions holds:\vskip 0cm
	(i)  $\al_1\al_2<1$ for
	\begin{equation*}
	\begin{split}
	& \al_1=\sup_{t\in [0,\omega]}\int_{t-m\omega}^t\sum_{i=1}^n \be_i(s)b_i(s)e^{-\be_i(s)N^*(s)}e^{-\int_s^t a(u)\, du}  ds \\
	&\al_2=\sup_{t\in [0,\omega]}\int_{t-m\omega}^t\sum_{i=1}^n \be_i(s)b_i(s)e^{-\int_s^t a(u)\, du}  ds;
	\end{split}
	\end{equation*}\vskip 0cm
	(ii) $\dps\sum_{i=1}^n \be_i(t)b_i(t)\le a(t),\, t\ge 0;$ \vskip 0cm
	(iii)
	$ \left(\overline {\be}\, \overline{ N^*}(e^{\overline {\be}\, \overline{ N^*}}-1)\right)^{1\over 2} 
	\Big (1- e^{-\bar{m} \int_0^{\omega} a(u)\, du}\Big)<1.$
\end{corollary}

\begin{corollary}\label{cor3.2}  For the DDE
	\begin{equation}\label{3.15}
	N'(t)+a(t)N(t)=b(t)e^{-N(t-m\omega)},\q t\ge 0,
	\end{equation}
	where $\omega>0, m\in\N$ and $a(t), b(t)$ are positive, $\omega$-periodic and continuous functions, there is a  positive $\omega$-periodic solution $N^*(t)$, which is a global attractor of all positive solutions if 
	$$\overline{N^*}\Big (1-e^{-m\int_0^\omega a(u)\, du}\Big )e^{-m\int_0^\omega a(u)\, du}
	\sup_{t\in [0,\omega]}\int_0^{m\omega}\!\! b(t+s)e^{\int_t^{t+s} a(u)\, du}\, ds<1.
	$$
\end{corollary}

\begin{proof} Here, we use Lemma \ref{lem3.3} and Corollary \ref{cor2.1} directly. The Yorke condition \eqref{2.2} is satisfied with
	$\la_1(t)=b(t)e^{-N^*(t)}$ and $\la_2(t)=b(t)$. For $\al_1(t),\al_2(t)$ given by \eqref{2.21} we obtain
	\begin{equation*}
	\begin{split}
	\al_1:=\al_1(t)&=\sup_{t\in [0,\omega]}\int_{t-m\omega}^t b(s)e^{-N^*(s)}e^{-\int_s^t a(u)\, du} \, ds\\
	&=\sup_{t\in [0,\omega]}\int_{t-m\omega}^t {d\over {ds}} \left [ N^*(s)e^{-\int _s^t a(u)\, du}\right]\, ds\\
	&=\sup_{t\in [0,\omega]}N^*(t)\Big (1-e^{-m\int_0^\omega a(u)\, du}\Big )=\overline{N^*}\Big (1-e^{-m\int_0^\omega a(u)\, du}\Big )
	\end{split}
	\end{equation*}
	and
	\begin{equation*}
	\begin{split}
	\al_2:=\al_2(t)&=\sup_{t\in [0,\omega]}\int_{t-m\omega}^t\!\!\! b(s)e^{-\int_s^t a(u)\, du}  ds\\
	&=e^{-m\int_0^\omega a(u)\, du}\sup_{t\in [0,\omega]}\int_0^{m\omega}\!\!\! b(t+s)e^{\int_t^{t+s} a(u)\, du} ds.
	\end{split}
	\end{equation*}
\end{proof}

\begin{remark}\label{rmk3.2} \rm{ Corollary \ref{cor3.2} is easily extended to \eqref{3.14}, however our interest here is to compare the statement in Corollary  \ref{cor3.2}  with   \cite{4.}. Using an iterative technique, Graef et al. \cite{4.}.  showed that the $\omega$-periodic positive solution $N^*(t)$  of \eqref{3.15} is globally attractive if 
		$$\sigma:=\int_0^{m\omega} b(s)e^{-N^*(s)} \, ds\le 1.$$
		With the notations of the above proof, clearly $\al_1<\sigma$. On the other hand, 
		observe that $\sigma =\int_0^{m\omega}  a(s)N^*(s)\, ds$; 
		whether $\al_2\le\sigma$ or not depends on the relative sizes of the functions $a(t),b(t)$ and the delay $m\omega$. Therefore, the criteria in Corollary \ref{cor3.2} and in  \cite{4.} are not comparable without further information on the coefficient functions and delay.}
\end{remark}

We now study the particular case of \eqref{3.1}  with $\tau_i(t)\equiv m_i\omega$ ($m_i$ positive integers) and  the impulses given by $I_k(u)=b_k u$:
\begin{equation}\label{3.16}
\begin{split}
&N'(t)+a(t)N(t)=\sum_{i=1}^n b_i(t)e^{-\be_i (t)N(t-m_i\omega)},\q 0\le t\ne t_k, \\
&\Delta N(t_k)=b_kN(t_k),\q k=1,2,\dots. 
\end{split}
\end{equation}
As before, we assume that $a(t),b_i(t),\be_i(t)$ are as in
($f_0$),  and that  ($i_0$) holds, i.e.,   $b_k>-1$ for $k\in\N$, $0<t_1<t_2<\cdots<t_p<\omega$ and 
\begin{equation}\label{3.17}t_{k+p}=t_k+\omega,\q b_{k+p}=b_k,\q \forall k\in\N, \end{equation}
where $p$ is some  positive integer. In this setting, assumption ($i_2$) translates as
\begin{equation}\label{3.18}\prod_{i=1}^p (1+b_k)\le 1. \end{equation}

Under these assumptions,  the existence of a  positive $\omega$-periodic solution  $N^*(t)$  of \eqref{3.16} follows from  \cite{6.}. The next theorem recovers the criterion  for its global asymptotic stability  established  in  \cite{6.}.

\begin{theorem}\label{thm3.4} 
	Consider system \eqref{3.16} with $b_k>-1$ for all $k$, and assume
	($f_0$), \eqref{3.17} and \eqref{3.18}. For  $N^*(t)$  a  positive $\omega$-periodic solution, whose existence is given in  \cite{6.}, $N^*(t)$ is globally attractive if 
	\begin{equation}\label{3.19}\sum_{i=1}^n \left (\prod_{k=1}^p (1+b_k)^{-m_i}\right) b_i(t)\be_i(t)\le a(t),\q t\ge 0. \end{equation}
\end{theorem}

\begin{proof} After translating the $\omega$-periodic solution $N^*(t)$ to the origin, we obtain system \eqref{3.7} with $\tilde I_k(x(t_k))=b_k x(t_k),\, k=1,2,\dots$. We now effect the change of variables  in \eqref{2.8}, 
	which in this situation reads as
	\begin{equation}\label{3.20}
	y(t)=\prod_{k:0\le t_k<t}(1+b_k)^{-1} x(t), \end{equation}
	and obtain an equivalent DDE with piecewise coefficients but no impulses, given by
	\begin{equation}\label{3.21}
	y'(t)+a(t)y(t)=\sum_{i=1}^n \tilde b_i(t)e^{-\be_i(t)N^*(t)}\left (e^{-\tilde \be_i(t)y(t-m_i\omega)}-1\right),
	\end{equation}
	where
	$$\tilde b_i(t)=b_i(t)\prod_{k:0\le t_k<t}(1+b_k)^{-1},\ \tilde \be_i(t)=\be_i(t) \prod_{k:0\le t_k<t-m_i\omega}(1+b_k),\q i\le i\le n.$$
	Reasoning as in Lemma \ref{lem3.3},   we deduce that for \eqref{3.21} the Yorke condition \eqref{2.2} holds with
	$\la_1(t)=\la_2(t)=\sum_{i=1}^n \tilde b_i(t) \tilde \be_i(t)$. Next, we observe that
	$$\sum_{i=1}^n \tilde b_i(t) \tilde \be_i(t)=\sum_{i=1}^n \left (\prod_{k=1}^p (1+b_k)^{-m_i}\right) b_i(t)\be_i(t),$$
	and apply Corollary \ref{cor2.2} to \eqref{3.21} to  obtain the result.\end{proof}

\begin{remark}\label{rmk3.3} \rm{Yan \cite{12.} considered model \eqref{3.16} with the following additional constraint:  the function
		$$t\mapsto \Theta (t):=\prod_{0\le t_k<t} (1+b_k)$$
		is $\omega$-periodic. 
		As pointed out by Liu and Takeuchi \cite{6.},  the condition of $ \Theta (t)$ being $\omega$-periodic is too restrictive:  it implies \eqref{3.17} and that   $\prod_{i=1}^p (1+b_k)=1$. 
		Under these assumptions, Yan \cite{12.} gave additional sufficient conditions for the existence and global attractivity of a unique $\omega$-periodic solution of \eqref{3.16}. Nevertheless, Liu and Takeuchi  remarked that Yan's  proof was not complete.  In turn,  they   proved themselves  the existence of a positive $\omega$-periodic solution $N^*(t)$ to \eqref{3.16} if
		$$\prod_{i=1}^p (1+b_k)e^{-\int_0^\omega a(s)\, ds}<1,$$
		which is always true if \eqref{3.18} is satisfied,
		and  showed that $N^*(t)$ is globally attractive if, in addition to
		\eqref{3.18}, the above condition \eqref{3.19} was imposed.  
		The technique in \cite{6.}   is quite different from the approach  in  \cite{12.}: the latter adapts the method in \cite{4.}, whereas the main idea in \cite{6.} is to assume  a Yorke-type condition of the form  \eqref{2.24}, and use  \cite{9.}. Although this scenario is  a particular situation of the general situation treated in our Corollary \ref{cor2.2}, the work in \cite{6.} was  a  strong motivation for the investigation carried out in this section.}\end{remark}


For the particular case of system \eqref{3.16} under conditions  ($f_0$), ($i_0$) and \eqref{3.18}, let $N^*(t)$ be a positive $\omega$-periodic  solution. Now, rather than \eqref{3.20}, we use an alternative  change of variables: 
\begin{equation}\label{3.22}
y(t)={{N(t)}\over{N^*(t)}}-1.
\end{equation}
This   leads to the equivalent system 
\begin{equation}\label{3.23}y'(t)=-r(t)y(t)+{1\over{N^*(t)}}\sum_{i=1}^n b_i(t)e^{-\be_i(t)N^*(t)}\Big [e^{-\be_i(t)N^*(t)y(t-m_i\omega)}-1\Big],\ 0\le t\ne t_k,
\end{equation}
with  $\Delta y(t_k) =0$ for $k\in\N$, where  $r(t)$  is a piecewise continuous functions defined by
$$r(t)={1\over{N^*(t)}}\sum_{i=1}^n b_i(t)e^{-\be_i(t)N^*(t)}.
$$  
In other words, \eqref{3.23} is  a DDE with piecewise continuous coefficients and no impulses.
As before, let $\bar{m}=\max_{1\le i\le n}m_i$. In this situation, the set of admissible initial conditions for \eqref{3.23} is
$S_1=\{\var \in C([-\bar{m}\omega,0]): \var (\th)\ge -1\ {\rm for} \ -\bar{m}\omega\le \th <0,\ \var (0)>-1\}.$

\med

\begin{remark}\label{rmk3.4} \rm{For the situation of \eqref{3.1} with $\tau_i(t)\equiv m_i\omega$ and impulsive functions $I_k$ satisfying $(i_1)$ with $b_k<a_k$ for some $k\in\{1,\dots,p\}$,  the change of variables \eqref{3.22} is not suitable for our purposes, since it transforms \eqref{3.1} into \eqref{3.23} together with the impulsive conditions
		$$
		\Delta y(t_k) =\hat I_k(y(t_k)),\q \, k\in\N,$$
		where $\hat I_k$ satisfy
		$$ {{1+b_k}\over {1+a_k}}x^2\leq x[x+\hat I_k(x)]\leq {{1+a_k}\over {1+b_k}}x^2,\q x\in\R,\q k\in\N.
		$$
		Therefore (H2)(i) is never satisfied, and the results in Section 2 cannot be invoked.}\end{remark}

We now apply Corollary \ref{cor2.1}    to the non-impulsive equation \eqref{3.23}.

\begin{theorem}\label{thm3.5} Consider system \eqref{3.16} with $b_k>-1$ for all $k$, assume
	($f_0$), \eqref{3.17}, \eqref{3.18}, and let $N^*(t)$ be an $\omega$-periodic positive solution. 
	If  $\al_1\al_2<1$, where
	\begin{equation}\label{3.24}
	\begin{split}
	\al_1&:=\sup _{t\in [0,\omega]} {1\over{N^*(t)}} \int_{t-\bar{m}\omega}^t\sum_{i=1}^n b_i(s)\be_i(s)N^*(s) e^{-\be_i(s)N^*(s)-\int_s^t a(u)\, du} 
	\prod_{k:s\le t_k<t} (1+b_k)\, ds\\
	\al_2&:=\sup _{t\in [0,\omega]} {1\over{N^*(t)}}\int_{t-\bar{m}\omega}^t \sum_{i=1}^n b_i(s)\be_i(s) N^*(s) e^{-\int_s^t a(u)\, du} 
	\prod_{k:s\le t_k<t} (1+b_k)\, ds,
	\end{split}
	\end{equation}
	where  $\dps \bar{m}=\max_{1\le i\le p} m_i$, then  $\lim_{t\to\infty} \big (N(t)-N^*(t)\big )=0$ for  any positive solution $N(t)$ of \eqref{3.16}.
\end{theorem}

\begin{proof} First, we observe that
	for $t>0, t\ne t_k,$
	$$r(t)={1\over{N^*(t)}}\Big [(N^*)'(t)+a(t)N^*(t)\Big].
	$$
	For $t>\bar{m}\omega$ and $t-\bar{m}\omega\le s\le t,$ we get
	\begin{equation*}
	\begin{split}
	\int_{s}^t r(u)\, du&=\int_s^t a(u)\, du+\int_s^t{{(N^*)'(u)}\over{N^*(u)}}\, du\\
	&=\int_s^t a(u)\, du+\log \left ( {{N^*(t)}\over {N^*(s^+)}}\right)+\sum_{k:s< t_k<t}\log \left ( {{N^*(t_k)}\over {N^*(t_{k}^+)}}\right)\\
	&=\int_s^t a(u)\, du+\log \left ({{N^*(t)}\over {N^*(s)}} \prod_{k:s\le t_k<t} (1+b_k)^{-1}\right),
	\end{split}
	\end{equation*}
	hence
	\begin{equation}\label{3.25}
	e^{-\int_s^t r(u)\, du}=e^{-\int_s^t a(u)\, du} \ {{N^*(s)}\over {N^*(t)}} \prod_{k:s\le t_k<t} (1+b_k).
	\end{equation}
	For the DDE without impulses \eqref{3.23}, the Yorke hypothesis (H4) is satisfied with
	$$\la_1(s)=\sum_{i=1}^n b_i(s)\be_i(s)e^{-\be_i(s)N^*(s)},\q \la_2(s)=  \sum_{i=1}^n b_i(s)\be_i(s).
	$$
	Now, let $\al_j (t)=\int_{t-\bar{m}\omega}^t \la_j(s) e^{-\int_s^t r(u)\, du}\, ds$ for $j=1,2, t>0$.
	The $\omega$-periodic functions $\al_j (t)$ satisfy $\al_j (t)\le \al_j$ for $\al_j, j=1,2,$ as in \eqref{3.24}, and the result follows from Corollary \ref{cor2.1}.\end{proof}

Following the approach  in Theorem \ref{thm3.3}, we now get  estimates for $\al_1,\al_2$ which are easier to verify, although they are not as refined as in \eqref{3.24}.

\begin{theorem}\label{thm3.6} Consider system \eqref{3.16} with $b_k>-1$ for all $k$ and denote $\dps \bar{m}=\max_{1\le i\le p} m_i$. Assume
	($f_0$), \eqref{3.17}, \eqref{3.18}, and let $N^*(t)$ be an $\omega$-periodic positive solution. 
	If  
	\begin{equation}\label{3.26}\sigma:=\Big (\overline{\be N^*}(e^{ \overline{\be N^*}}-1)\Big )^{1\over 2} \left [1-  \left (e^{-\int_0^\omega a(u)\, du}\prod_{i=1}^p (1+b_k)\right)^{\bar{m}}\right]<1,
	\end{equation}
	then  $\lim_{t\to\infty} \big (N(t)-N^*(t)\big )=0$ for  any positive solution $N(t)$ of \eqref{3.16}. \end{theorem}

\begin{proof} We first observe that, for any fixed $t>0$, the function  $h(s):=e^{-\int_s^t r(u)\, du}$ defined for $ s\in [t-\bar{m}\omega,t]$ is continuous.
	Applying Theorem \ref{thm3.3} to \eqref{3.16}, we deduce  the global asymptotic stability of $N^*(t)$ if $\sigma <1$, where $\sigma$ is given by \eqref{3.8} with $b_k=0$ for all $k$ and $a(t)$ is replaced by $r(t)$.  Thus $\sigma$ is as in \eqref{3.26}, since \eqref{3.25} implies
	$$e^{-\bar m\int_0^\omega r(u)\, du}=e^{-\bar m\int_0^\omega a(u)\, du}\prod_{k=1}^p (1+b_k)^{\bar m}.$$
\end{proof}

\begin{remark}\label{rmk3.5} \rm{
		Under  stronger constraints, including that $\prod_{i=1}^p (1+b_k)=1$,  Yan \cite{12.} claimed that the condition
		$$\sum_{i=1}^n \overline{q_i}\int_0^{\bar{m}\omega} b_i(s)\Big(\prod_{k:0<t_k<s}(1+b_k) \Big) e^{-\be_i(s)N^*(s))}\, ds\le 1,$$
		with $\overline{q_i}=\dps \sup_{t\ge 0}\, \Big(\prod_{k:0<t_k<t}(1+b_k) \Big) \be_i(t)$, implies the global attractivity of the positive $\omega$-periodic solution $N^*(t)$ of \eqref{3.16} (see also Remark \ref{rmk3.3}, for a comment in   \cite{6.}).
		In any case, the results  in Theorems \ref{thm3.5} and \ref{thm3.6} are  not easily comparable with the claim in \cite{12.}.}\end{remark}

\section*{Acknowledgments}\noindent{This work was  supported by Funda\c c\~ao para a Ci\^encia e a Tecnologia, under 
	projects  UID/MAT/04561/2013 (T. Faria) and\\ PEstOE/MAT/UI0013/2014 (J.J. Oliveira).


\medskip
Received xxxx 20xx; revised xxxx 20xx.
\medskip

\end{document}